\documentclass[12pt,a4paper]{amsart}
\usepackage{tikz,xcolor}
\newtheorem{theo+}              {Theorem}           [section]
\newtheorem{prop+}  [theo+]     {Proposition}
\newtheorem{coro+}  [theo+]     {Corollary}
\newtheorem{lemm+}  [theo+]     {Lemma}
\newtheorem{exam+}  [theo+]     {Example}
\newtheorem{rema+}  [theo+]     {Remark}
\newtheorem{defi+}  [theo+]     {Definition}
\def \r{\mbox{${\mathbb R}$}}

\newenvironment{theorem}{\begin{theo+}}{\end{theo+}}
\newenvironment{proposition}{\begin{prop+}}{\end{prop+}}
\newenvironment{corollary}{\begin{coro+}}{\end{coro+}}

\usepackage{amsthm}
\theoremstyle{plain} \theoremstyle{remark}
\newtheorem{remark}{Remark}
\newtheorem{example}{Example}

\def \r{\mbox{${\mathbb R}$}}

\title{On conformal biharmonic hypersurfaces}
\author{A. Mohammed Cherif and Ye-Lin Ou}

\address{Faculty of Exact Sciences,\newline\indent University Mustapha Stambouli,
\newline\indent Mascara, Algeria.\newline\indent Email: a.mohammedcherif@univ-mascara.dz
\newline\indent \newline\indent
Department of Mathematics,\newline\indent
East Texas A $\&$ M University,\newline\indent Commerce, TX
75429, U S A.\newline\indent E-mail:yelin.ou@etamu.edu}
\begin{document}

\title[Biharmonic conformal hypersurfaces]{On biharmonic conformal hypersurfaces}

\date{01/06/26}
\maketitle

\section*{Abstract}
In this paper, we first derive biharmonic equation for conformal hypersurfaces in a generic Riemannian manifold generalizing that for biharmonic hypersurfaces in \cite{Ou1} and that for biharmonic conformal surfaces in \cite{Ou3, Ou2, Ou4}. We then show that if a totally umbilical hypersurface in a space form admits a biharmonic conformal immersion into the ambient space, then the conformal factor has to be an isoparametric function. We also prove that no part of a  non-minimal totally umbilical hypersurface in a space form of nonpositive curvature  admits a biharmonic conformally immersion into that space form whilst, for the positive curvature space form, we show that the totally umbilical hypersurface $S^4(\frac{\sqrt{3}}{2})\hookrightarrow S^5$ does admit a biharmonic conformal immersion into $S^5$. 
\begin{quote}
{\footnotesize }
\end{quote}

\section{Biharmonic conformal hypersurfaces: Equations and some examples}

Recall that a map $\varphi:(M, g)\longrightarrow (N, h)$ between Riemannian manifolds is biharmonic if it is a critical point of the
bienergy functional
\begin{equation}\nonumber
E_{2}\left(\varphi,\Omega \right)= \frac{1}{2} {\int}_{\Omega}
\left|\tau(\varphi) \right|^{2}{\rm d}x
\end{equation}
for any compact subset $\Omega$ of $M$, where $\tau(\varphi)={\rm
Trace}_{g}\nabla {\rm d} \varphi$ denotes the tension field of $\varphi$.
Biharmonic map equation is given by the Euler-Lagrange equation of the bienergy functional, which is a system of fourth order  PDEs (see \cite{Ji1})
\begin{equation}\notag
\tau_{2}(\varphi):={\rm
Trace}_{g}(\nabla^{\varphi}_{\cdot}\nabla^{\varphi}_{\cdot}-\nabla^{\varphi}_{\nabla^{M}_{\cdot}{\cdot}})\tau(\varphi)
- {\rm Trace}_{g} R^{N}({\rm d}\varphi({\cdot}), \tau(\varphi)){\rm d}\varphi({\cdot})
=0,
\end{equation}
where  $R^{N}$ to denotes
the curvature operator of $(N, h)$ defined by
$$R^{N}(X,Y)Z=
[\nabla^{N}_{X},\nabla^{N}_{Y}]Z-\nabla^{N}_{[X,Y]}Z.$$

An important aspect of the  study of biharmonic maps focuses on geometric  biharmonic maps, i.e., biharmonic maps with  geometric  meaning (or under some geometric constraints). These include biharmonic isometric immersions (i.e.,  biharmonic submanifolds  which generalize the concept  of minimal submanifolds) or  biharmonic conformal maps, both of these have led to  fruitful results and interesting links. For the study of biharmonic submanifolds see survey 
articles \cite{O2, Ou6}, a  recent book \cite{OC}, and the vast  references therein. For biharmonic Riemannian submersions, or more generally, horizontally weakly conformal biharmonic maps, see \cite{ Ou0, LO1, O1, BFO, LO2, WO1, 
GO, AO,  U1, U2, Ou5, WO2, WO3, Ou7, MS}. For biharmonic conformal maps between manifolds of the same dimension see \cite{BK, BFO, BLO,  BO, Ou8}. For biharmonic conformal immersions see \cite{Ou3, Ou2, Ou4, WC}.

Recall that a conformal immersion is an immersion $\phi: (M^{m},\bar{g}) \to (N^n,h)$ between two Riemannian manifolds such that $\phi^*h = \lambda^2 \bar{g}$ for a positive function $\lambda$ on $M$. When $\lambda$ is a constant, it is called a homothetic immersion. In particular, when $\lambda\equiv 1$ the conformal immersion becomes an isometric immersion. Recall also  that the image of an isometric immersion is called a submanifold, and a submanifold of codimension one is called a hypersurface. Similarly, the image of a conformal immersion is called a conformal submanifold and a codimension one conformal submanifold is called a conformal hypersurface.

Note that every conformal immersion $\phi: (M^{m},\bar{g}) \to (N^n,h)$ with $\phi^*h = \lambda^2 \bar{g}$ is associated with a submanifold, i.e., an isometric immersion $\phi: (M^{m}, g=\lambda^2\bar{g}) \to (N^n,h)$. Conversely, any submanifold (i.e., an isometric immersion) $(M^m, \phi^*h)\to (N^n, h)$ and any positive function $\lambda$ on $M$ is associated to a conformal immersion $(M^m,\lambda^{-2} \phi^*h)\to (N^n, h)$ with conformal factor $\lambda$. Following \cite{Ou2} ( Definition 2.2), we say a hypersurface in a Riemannian manifold $(N^{m+1}, h)$ defined by an isometric immersion $\phi: (M^{m},g=\phi^*h) \to (N^{m+1},h)$ admits a {\bf biharmonic conformal immersion} into $(N^{m+1}, h)$, if there exists  a smooth function $\lambda: M^m\to \r^+$ such that the conformal immersion $\phi: (M^{m}, \overline{g}=\lambda^{-2}g) \to (N^{m+1},h)$ with conformal factor $\lambda$ is a biharmonic map.

 It follows that a biharmonic conformal immersion with conformal factor $\lambda$ and  a hypersurface admitting a biharmonic conformal immersion with $\lambda$ satisfy the same equation. For $m=2$, it was proved in \cite{Ou3, Ou2, Ou4} that a conformal immersion $\phi : (M^{2},{\bar g}) \to (N^3,h)$ with
$\varphi^{*}h=\lambda^2{\bar g}$ is biharmonic if and only if 
\begin{equation}\label{M03}
\begin{cases}
\Delta (\lambda^2 H )-(\lambda^2H)[|A|^2-{\rm Ric}^N(\xi,\xi)]=0,\\A({\rm grad} (\lambda^2 H))+ (\lambda^2 H) [{\rm grad}
H- \,({\rm Ric}^N\,(\xi))^{\top}]=0,
\end{cases}
\end{equation}
where $\xi$, $A$, and $H$ are the unit normal vector field, the
shape operator, and the mean curvature function  of
the surface $\varphi(M)\subset (N^3, h)$ respectively, and the
operators $\Delta,\; {\rm grad}$ and $|,|$ are taken with respect to
the induced metric $g=\varphi^{*}h=\lambda^2{\bar g}$ on the surface.

For  general $m$, we have
\begin{theorem}\label{th1}
A conformal immersion $\phi : (M^m, \overline{g})\to (N^{m+1}, h)$ with $\phi^*h=\lambda^2\bar{g}$ is biharmonic if and only if
\begin{align}\label{NE}
&m\Big[\Delta H-H|A|^2+H{\rm Ric}^N(\xi,\xi)\Big]-2(m-2){\rm Tr\,}g(A(\cdot),{\nabla}^M_{\cdot}\nabla \ln\lambda)\\\notag
&+2mH\Big[\Delta {\rm ln}\lambda-(m-4)\left|\nabla\ln\lambda\right|^2\Big]-2m(m-4)g(\nabla\ln\lambda,\nabla H)\\\notag
&+(m-2)(m-6)g(A(\nabla\ln\lambda),\nabla\ln\lambda)=0,\\\label{TE}
&-m\Big[2A(\nabla H)+\frac{m}{2}\nabla H^2-2H({\rm Ric}^N\xi)^\top\Big]-(m-2)\Big[2\, {\rm Ric}^M\,(\nabla \ln\lambda)\\\notag
&+\nabla {\Delta}\ln\lambda\Big]-2(m-2)\Big[\Delta {\rm ln}\lambda-(m-4)\left|\nabla\ln\lambda\right|^2\Big]\nabla\ln\lambda\\\notag
&+\frac{1}{2}(m-2)(m-6)\nabla |\nabla\ln\lambda|^2+2m(m-4)HA(\nabla\ln\lambda)=0,
\end{align}
where $\xi$, $A$, and $H$ are the unit normal vector field, the
shape operator, and the mean curvature function  of
the hypersurface $\phi(M)\subset (N^{m+1}, h)$ respectively, and the
operators $\Delta,\; \nabla$ and $|,|$ are taken with respect to
the induced metric $g=\phi^{*}h=\lambda^2 {\bar g}$ on the hypersurface.

\end{theorem}

\begin{proof}
It was proved in \cite{Ou3} that a conformal immersion $\phi : (M^m, \bar{g})\to (N^n, h)$ with $\phi^*h=\lambda^2\bar{g}$ is biharmonic if and only if
\begin{align}\label{T2i}
&\tau_{2}(\phi, g)+(m-2)J^{\phi}_{g}({\rm d}{\phi}({\nabla\ln}\lambda))\\\notag 
& +2\Big[\Delta {\rm ln}\lambda-(m-4)\left|{\nabla\ln}\lambda\right|^2\Big]\tau(\phi,g)
-(m-6)\nabla^{ \phi}_{{\nabla\ln}\lambda}\,\tau(\phi,g)\\\notag 
& -2(m-2)\Big[\Delta {\rm ln}\lambda-(m-4)\left|{\nabla\ln}\lambda\right|^2\Big]{\rm d}{\phi}({\nabla\ln}\lambda)\\\notag 
&+(m-2)(m-6)\nabla^{ \phi}_{{\nabla\ln}\lambda}\,{\rm d}{\phi}({\nabla\ln}\lambda)=0,
\end{align}
where $\tau_{2}(\phi, g)$ is the bitension field of the hypersurface $\phi(M)\subset (N^{m+1}, h)$, and  $\nabla$ and $\Delta$ denote the gradient and the Laplacian taken with respect to the metric $g=\lambda^2 {\bar g}$.

Since $\phi$ is an immersion, we can locally identify $p\in M$ with $\phi(p)\in N$, and ${\rm d}\phi(X)=X$. We will use these to compute the term $J^{ \phi}_{g}({\rm d} \phi(\nabla\ln\lambda))$. Let $\{\overline{e}_i\}_{1\leq i\leq m}$ be a local orthonormal frame in $M$ with respect to the Riemannian metric $\overline{g}$. Then, $\{e_i=\frac{1}{\lambda}\overline{e}_i\}_{1\leq i\leq m}$ is a local orthonormal frame in $M$ with respect to the Riemannian metric $g=\lambda^2\overline{g}$. From the definition of the Jacobi operator $J^{ \phi}_{g}$, we have
\begin{eqnarray}\label{eq2}
  J^{ \phi}_{{g}}({\rm d}{\phi}({\nabla\ln}\lambda))
   &=& -\sum_{i=1}^m\big[ R^N(\nabla\ln\lambda,{e}_i){e}_i
       + \nabla^\phi_{{e}_i}\nabla^\phi_{{e}_i}\nabla\ln\lambda
       - \nabla^\phi_{{\nabla}^M_{{e}_i}{e}_i}\nabla\ln\lambda\big].
\end{eqnarray}

Denoting $T_1 = -\sum_{i=1}^m R^N(\nabla\ln\lambda,{e}_i){e}_i, \; T_2 = - \sum_{i=1}^m\nabla^\phi_{{e}_i}\nabla^\phi_{{e}_i}\nabla\ln\lambda$ and \\$T_3 = \sum_{i=1}^m\nabla^\phi_{{\nabla}^M_{{e}_i}{e}_i}\nabla\ln\lambda
$, we will compute these three terms as follows.\\

By using the Gauss equation
\begin{eqnarray*}
  {g}({R}^M(V,W)X,Y)
   &=& h(R^N(V,W)X,Y)+h(B(W,X),B(V,Y))\\
   & &   -h(B(V,X),B(W,Y)),
\end{eqnarray*}
 we have the tangential part of $T_1$ is given by
\begin{eqnarray}\notag
  T_1^\top
   &=& -\sum_{i,j=1}^m h(R^N(\nabla\ln\lambda,{e}_i){e}_i,{e}_j){e}_j\\ \label{eq3}
   &=& -\sum_{i,j=1}^m\Big[{g}({R}^M(\nabla \ln\lambda,{e}_i){e}_i,{e}_j)
      +h(B(\nabla \ln\lambda,{e}_i),B({e}_i,{e}_j))\\
   &&\nonumber    -h(B({e}_i,{e}_i),B(\nabla \ln\lambda,{e}_j))\Big]{e}_j.
\end{eqnarray}

Using $B(X,Y)=g(A(X),Y)\xi,   \sum_{i=1}^mB({e}_i,{e}_i)=mH\xi, 
  $\\ $ {\rm Ric}^M\,(\nabla \ln\lambda)=\sum_{i=1}^m{R}^M(\nabla \ln\lambda,{e}_i){e}_i $ we can rewrite Equation \eqref{eq3} as
\begin{eqnarray}\label{eq4}
  T_1^\top
   &=&-{\rm Ric}^M\,(\nabla \ln\lambda)
      -\sum_{j=1}^mg(A(\nabla \ln\lambda),A({e}_j)){e}_j
      +mHA(\nabla \ln\lambda).
\end{eqnarray}

The normal part of $T_1$ is given by
\begin{align}\label{eq5}
  T_1^\bot = -\sum_{i=1}^mh(R^N(\nabla \ln\lambda,{e}_i){e}_i,\xi)\xi  =-{\rm Ric}^N(\nabla \ln\lambda,\xi)\xi.
\end{align}

From  \eqref{eq4} and \eqref{eq5}, we have
\begin{eqnarray}\label{T10}
  T_1 &=&\nonumber-{\rm Ric}^M\,(\nabla \ln\lambda)
      -\sum_{j=1}^mg(A(\nabla \ln\lambda),A({e}_j)){e}_j
      +mHA(\nabla \ln\lambda)\\
   & &-{\rm Ric}^N(\nabla \ln\lambda,\xi)\xi.
\end{eqnarray}

The tangential part of $T_2$ is given by
\begin{eqnarray}\label{eq8}
  T_2^\top
    &=&\nonumber - \sum_{i,j=1}^mh(\nabla^\phi_{{e}_i}\nabla^\phi_{{e}_i}\nabla \ln\lambda,{e}_j){e}_j\\
    &=&\nonumber -  \sum_{i,j=1}^m\big[{e}_ih(\nabla^\phi_{{e}_i}\nabla \ln\lambda,{e}_j)
                 - h(\nabla^\phi_{{e}_i}\nabla \ln\lambda,\nabla^\phi_{{e}_i}{e}_j)\big] e_j\\
    &=&\nonumber - \sum_{i,j=1}^m\Big[ {e}_ih({\nabla}^M_{{e}_i}\nabla \ln\lambda,{e}_j)
                 - h({\nabla}^M_{{e}_i}\nabla \ln\lambda,{\nabla}^M_{{e}_i}{e}_j)\\
    & &          - h(B({e}_i,\nabla \ln\lambda),B({e}_i,{e}_j))\Big] e_j.
\end{eqnarray}
Using $\phi^*h=\lambda^2\overline{g}=g, d\phi(X)=X$, we can write Equation \eqref{eq8}  as
   \begin{eqnarray}\label{eq9}
  T_2^\top
    &=&\nonumber  -  \sum_{j=1}^m\Big[\sum_{i=1}^m\Big({e}_i{g}({\nabla}^M_{{e}_i}\nabla \ln\lambda,{e}_j)
                  - {g}({\nabla}^M_{{e}_i}\nabla \ln\lambda,{\nabla}^M_{{e}_i}{e}_j)\Big)\\
    & &\nonumber  - g(A(\nabla \ln\lambda),A({e}_j))\Big] e_j\\
    &=&           -  \sum_{j=1}^m\Big[\sum_{i=1}^m{g}({\nabla}^M_{{e}_i}{\nabla}^M_{{e}_i}\nabla \ln\lambda,{e}_j)
                  - g(A(\nabla \ln\lambda),A({e}_j))\Big] e_j.
  \end{eqnarray}
  
The normal part of $T_2$ is given by
\begin{eqnarray}\notag
  T_2^\bot
    &=&\nonumber  -\sum_{i=1}^m\big[\nabla^\phi_{{e}_i}\nabla^\phi_{{e}_i}\nabla \ln\lambda\big]^\bot\\\notag
    &=&\nonumber  -\sum_{i=1}^m\big[\nabla^\phi_{{e}_i}{\nabla}^M_{{e}_i}\nabla \ln\lambda
                  +\nabla^\phi_{{e}_i}B({e}_i,\nabla \ln\lambda)\big]^\bot\\\notag
    &=&           - \sum_{i=1}^m\Big[B( {e}_i,{\nabla}^M_{{e}_i}\nabla \ln\lambda)
                  + \nabla^\bot_{{e}_i}B({e}_i,\nabla \ln\lambda)\Big].
  \end{eqnarray}
By using $\nabla^\bot_{{e}_i}\xi=0$ for all $i=1,...,m$, we find that
\begin{eqnarray}\label{eq10*}
  T_2^\bot
    &=&\nonumber  - \sum_{i=1}^m\Big[ g(A({e}_i),{\nabla}^M_{{e}_i}\nabla \ln\lambda)
                  + \nabla^\bot_{{e}_i}g(A({e}_i),\nabla \ln\lambda)\Big] \xi\\
    &=&           -\sum_{i=1}^m\Big[2 g(A({e}_i),{\nabla}^M_{{e}_i}\nabla \ln\lambda)
                  + g({\nabla}^M_{{e}_i}A({e}_i),\nabla \ln\lambda)\Big]\xi.
  \end{eqnarray}
From \eqref{eq9} and \eqref{eq10*}, we have
\begin{eqnarray}\label{eq11}
  T_2
    &=&\nonumber - \sum_{i=1}^m\Big[ {\nabla}^M_{{e}_i}{\nabla}^M_{{e}_i}\nabla \ln\lambda
                  - g(A(\nabla \ln\lambda),A({e}_i)){e}_i\\
    & &           +2 g(A({e}_i),{\nabla}^M_{{e}_i}\nabla \ln\lambda)\xi
                  + g({\nabla}^M_{{e}_i}A({e}_i),\nabla \ln\lambda)\xi\Big].
  \end{eqnarray}
By a direct calculation, we find that
\begin{eqnarray}\label{eq12}
  T_3 &=&\sum_{i=1}^m \nabla^\phi_{{\nabla}^M_{{e}_i}{e}_i}\nabla \ln\lambda= \sum_{i=1}^m\Big[{\nabla}^M_{{\nabla}^M_{{e}_i}{e}_i}\nabla \ln\lambda
                   +B({\nabla}^M_{{e}_i}{e}_i,\nabla \ln\lambda)\Big]\\\notag
      &=&       \sum_{i=1}^m\Big[ {\nabla}^M_{{\nabla}^M_{{e}_i}{e}_i}\nabla \ln\lambda
                   +g(A({\nabla}^M_{{e}_i}{e}_i),\nabla \ln\lambda)\xi\Big].
\end{eqnarray}
By substituting \eqref{T10}, \eqref{eq11} and \eqref{eq12} in \eqref{eq2}, and using the formulas
\begin{eqnarray*}
 {\rm Tr\,}({\nabla}^M)^2\nabla \ln\lambda  &=& {\rm Ric}^M\,(\nabla \ln\lambda)+\nabla {\Delta}\ln\lambda,\\
\sum_{i=1}^m (\nabla_{e_i}A)(e_i)&=&m\,\nabla H-({\rm Ric}^N\xi)^\top,\; (see\; e.g., [12])
  \end{eqnarray*}
 we obtain 
\begin{eqnarray}\label{eq13}
  J^{ \phi}_{{g}}({\rm d}{\phi}({\nabla\ln}\lambda))
   &=&\nonumber -2 {\rm Ric}^M\,(\nabla \ln\lambda)
                +mHA(\nabla \ln\lambda)
                - \nabla {\Delta}\ln\lambda\\
   & &          -2\sum_{i=1}^m g(A({e}_i),{\nabla}^M_{{e}_i}\nabla \ln\lambda)\xi
                -mg(\nabla H,\nabla \ln\lambda)\xi.
\end{eqnarray}
On the other hand, using $\tau(\phi,g)=mH\xi$ and a straightforward computation we have
\begin{eqnarray}\label{eq14}
  \nabla^{ \phi}_{\nabla\ln\lambda}\,\tau(\phi,g)
    &=&\nonumber m\nabla^{\phi}_{\nabla\ln\lambda}H\xi  \\
    &=&\nonumber mg(\nabla H,\nabla \ln\lambda)\xi
                 +mH\nabla^{\phi}_{\nabla\ln\lambda}\xi\\
    &=&   mg(\nabla H,\nabla \ln\lambda)\xi
          -mHA(\nabla \ln\lambda).
\end{eqnarray}
 Substituting \eqref{eq13}, \eqref{eq14}, and 
\begin{eqnarray}\notag
\nabla^{ \phi}_{\nabla\ln\lambda} d \phi(\nabla\ln\lambda)
   &=&\nonumber \nabla^{M}_{\nabla\ln\lambda}\nabla\ln\lambda
                +B(\nabla\ln\lambda,\nabla\ln\lambda)\\\notag
   &=& \frac{1}{2}\nabla |\nabla\ln\lambda|^2
       +g(A(\nabla\ln\lambda),\nabla\ln\lambda)\xi,
\end{eqnarray}
into \eqref{T2i}, we conclude that the  conformal immersion $\phi : (M^m, \overline{g})\to (N^{m+1}, h)$ is biharmonic if and only if
\begin{align}\label{16}
&\tau_{2}(\phi, g)+(m-2)\Big[-2{\rm Ric}^M (\nabla \ln\lambda)+mHA(\nabla \ln\lambda)- \nabla {\Delta}\ln\lambda\\\notag
&-2\sum_{i=1}^mg(A({e}_i),{\nabla}^M_{{e}_i}\nabla \ln\lambda)\xi-mg(\nabla H,\nabla \ln\lambda)\xi\Big]
 +2mH\Big[\Delta {\rm ln}\lambda\\\notag
&-(m-4)\left|\nabla\ln\lambda\right|^2\Big]\xi-m(m-6)\Big[g(\nabla H,\nabla \ln\lambda)\xi-HA(\nabla \ln\lambda)\Big]\\\notag
& -2(m-2)\Big[\Delta {\rm ln}\lambda-(m-4)\left|\nabla\ln\lambda\right|^2\Big]\nabla\ln\lambda\\\notag
&+(m-2)(m-6)\Big[\frac{1}{2}\nabla |\nabla\ln\lambda|^2
+g(A(\nabla\ln\lambda),\nabla\ln\lambda)\xi\Big]=0.
\end{align}
Theorem \ref{th1} follows from \eqref{16} and the following formulas (see \cite{Ou1}) of bitension field of the hypersurface $\phi : (M^m, g=\lambda^2{\bar g})\to (N^{m+1}, h)$
\begin{eqnarray*}
 \tau_{2}(\phi, g)^\bot  &=&  m\Big[\Delta H-H|A|^2+H{\rm Ric}^N(\xi,\xi)\Big]\xi\\
 \tau_{2}(\phi, g)^\top  &=& -m\Big[2A(\nabla H)+\frac{m}{2}\nabla H^2-2H({\rm Ric}^N\xi)^\top\Big].
\end{eqnarray*}
\end{proof}

\begin{remark}
(i) It is easily seen that when $\lambda = 1$, the biharmonic conformal hypersurface equations in Theorem \ref{th1} simply reduce to biharmonic hypersurface equation obtained  in \cite{Ou2}. (ii) One can also check that when $m=2$ the biharmonic equations given in Theorem \ref{th1} reduce to (\ref{M03}).
\end{remark}

It is well known that a  conformal immersion from a $2$-dimensional manifold is minimal if and only if it is harmonic which is always biharmonic. For the domain dimension  $m\ge 3$ a minimal conformal immersion is no longer harmonic, and we have
\begin{corollary}
A minimal conformal immersion $\phi : (M^m, \overline{g})\to (N^{m+1}, h)\;(m\ge 3)$ with $\phi^*h=\lambda^2\bar{g}$ is biharmonic if and only if
\begin{align}\label{normal-part-hyp}
&2\lambda\,{\rm Tr\,}g(A(\cdot),{\nabla}^M_{\cdot}\nabla \lambda)-(m-6)g(A(\nabla\lambda),\nabla\lambda)=0,\\\label{Min-T}
&2\lambda\, {\rm Ric}^M\,(\nabla \lambda)+\nabla (\lambda\Delta\lambda)-\frac{m-4}{2}\nabla |\nabla \lambda|^2-(m-2)\lambda^{-1}|\nabla \lambda|^2\nabla \lambda=0.
\end{align}
In particular, a totally geodesic hypersurface $\phi : (M^m,g)\to (N^{m+1}, h)$ $(m\ge  3)$  can be biharmonically conformally immersed into $ (N^{m+1}, h)$ with a conformal factor $\lambda$ if and only if
\begin{align}\label{TG}
2\lambda\, {\rm Ric}^M\,(\nabla \lambda)+\nabla (\lambda\Delta\lambda)-\frac{m-4}{2}\nabla |\nabla \lambda|^2-(m-2)\lambda^{-1}|\nabla \lambda|^2\nabla \lambda=0.
\end{align}
\end{corollary}

\begin{proof}
By definition, the conformal immersion $\phi : M^m\to (N^{m+1}, h)\;  (m\ge  3)$ is minimal, then $H=0$. Substituting this into (\ref{NE}) and (\ref{TE}) we see that a minimal conformal immersion is  biharmonic if and only if
\begin{align}\notag
&2{\rm Tr\,}g(A(\cdot),{\nabla}^M_{\cdot}\nabla \ln\lambda)-(m-6)g(A(\nabla\ln\lambda),\nabla\ln\lambda)=0,\\\notag
&\Big[2\, {\rm Ric}^M\,(\nabla \ln\lambda)+\nabla {\Delta}\ln\lambda\Big]+2\Big[\Delta {\rm ln}\lambda-(m-4)\left|\nabla\ln\lambda\right|^2\Big]\nabla\ln\lambda\\\notag
&-\frac{(m-6)}{2}\nabla |\nabla\ln\lambda|^2=0.
\end{align}

By using $\Delta \ln \lambda=\lambda^{-1}\Delta \lambda-\lambda^{-2}|\nabla\lambda|^2, |\nabla\ln\lambda|^2=\lambda^{-2}|\nabla\lambda|^2$ and a straightforward calculation we obtain the first statement of the corollary. The second statement follows from the first one with $A=0$.
\end{proof}



Recall that a smooth function $f : M \to \mathbb{R}$ is called an \textbf{isoparametric function} if there exist smooth functions $\alpha, \beta : \mathbb{R} \to \mathbb{R}$ such that $|\nabla f|^2 = \alpha(f)$ and $\Delta f = \beta(f)$. It was proved in \cite{BFO} that a conformal map $\phi: (M^m, g)\to (N^m, h)$ with $\phi^*h=\lambda^2 g$ and $m\ne 4$ from an Einstein manifold is biharmonic if and only if $\lambda$ is an isoparametric function.

Our next corollary shows that the conformal factor of a totally geodesic biharmonic conformal immersion into a space form is always an isoparametric function, and in this case the biharmonic equation reduces to an ODE, which helps to construct many examples of biharmonic conformal hypersurfaces.

\begin{corollary}\label{co3}
A totally geodesic hypersurface $\phi:(M^m,g)\rightarrow(N^{m+1}(c),h)$ of a space form of dimension $m\geq3$ admits a biharmonic conformal immersion into the ambient space with
a nonconstant conformal factor $\lambda$ if and only if  $\lambda$ is an isoparametric function on $(M,g)$ with $|\nabla \lambda|^2=\alpha(\lambda)$ and $\Delta\lambda=\beta(\lambda)$ for some smooth real functions $\alpha$ and $\beta$ solving the ODE
\begin{align}\label{TGCM}
\lambda \beta'(\lambda)+\beta(\lambda)-\frac{m-4}{2}\alpha'(\lambda)-(m-2)\lambda^{-1}\alpha(\lambda)+2(m-1)c\,\lambda=0.
\end{align}
In particular, for $m=4$, the equation reduces to
\begin{align}\label{TGC4}
\lambda \beta'(\lambda)+\beta(\lambda)-2\lambda^{-1}\alpha(\lambda)+6c\,\lambda=0.
\end{align}
\end{corollary}
\begin{proof}
It is also easily checked that a totally geodesic hypersurface in a space form $(N^{m+1}(c),h)$ is a space form of the same constant sectional curvature $c$, and hence $ {\rm Ric}^M\,(\nabla \lambda)=(m-1)c\,\nabla  \lambda$. It follows from these and  (\ref{TG}) that a totally geodesic hypersurface $\phi : (M^m,g)\to (N^{m+1}(c), h)\;  (m\ge  3)$ can be  biharmonic conformally immersed into $ (N^{m+1}(c), h)$ if and only if
\begin{align}\label{TGCC}
2(m-1)c\,\lambda\,\nabla \,\lambda+\nabla (\lambda\Delta\lambda)-\frac{m-4}{2}\nabla |\nabla \lambda|^2-(m-2)\lambda^{-1}|\nabla \lambda|^2\nabla  \lambda=0.
\end{align}
In particular, if $m=4$,  a totally geodesic hypersurface $\phi : (M^4,g)\to (N^{5}, h)$ can be biharmonic conformally immersed into $ (N^5, h)$ if and only if
\begin{align}\label{4TG}
6c\,\lambda\,\nabla \lambda+\nabla (\lambda\Delta\lambda)-2\lambda^{-1}|\nabla \lambda|^2\nabla  \lambda=0.
\end{align}
Now,  (\ref{TGCC}) can be written as
\begin{align}\notag
\nabla \,\Big[(m-1)c\,\lambda^2+ \lambda\Delta\lambda-\frac{m-4}{2} |\nabla \lambda|^2\Big]=(m-2)\lambda^{-1}|\nabla \lambda|^2\nabla  \lambda.
\end{align}
By using this and Lemma 2 of \cite{BO} we have

\begin{align}\label{TG20}
(m-1)c\,\lambda^2+ \lambda\Delta\lambda-\frac{m-4}{2} |\nabla \lambda|^2=u(\lambda),
\end{align}
for some smooth function $u$ with 
\begin{align}\label{TG21}
u'(\lambda)=\lambda^{-1}|\nabla \lambda|^2.
\end{align}
It follows from (\ref{TG21}) and (\ref{TG20}) that $\lambda$ is an isoparametric function. By substituting $|\nabla \lambda|^2 = \alpha(\lambda)$ and $\Delta \lambda = \beta(\lambda)$ into  (\ref{TGCC}) and a straightforward computation yields 
\begin{align}\notag
\Big[\lambda \beta'(\lambda)+\beta(\lambda)-\frac{m-4}{2}\alpha'(\lambda)-(m-2)\lambda^{-1}\alpha(\lambda)+2(m-1)c\,\lambda\Big]\nabla\lambda=0,
\end{align}
which gives (\ref{TGCM}) since $\lambda$ is not constant. Clearly,  (\ref{TGCM}) reduces to   (\ref{TGC4}) when $m=4$.
\end{proof}

Now  we will show that many  examples of proper biharmonic conformal immersions of totally geodesic hypersurfaces into a space form can be constructed by using Corollary \ref{co3}.

\begin{example}(\em Proper biharmonic conformal hypersurfaces in a Euclidean space)\\
Let $a_1,...,a_m,b\in\mathbb{R}$. Take
$$\Omega=\{(x_1,...,x_{m})\in\mathbb{R}^{m}\,|\,\sum_{i=1}^m a_ix_i+b>0\}.$$
We consider the totally geodesic hypersurface $\phi:\Omega\rightarrow(\mathbb{R}^{m+1},h)$ defined by
$$\phi(x_1,...,x_m)=(x_1,...,x_m, \sum_{i=1}^m a_ix_i+b),$$
where $h=dy_1^2+...+dy_{m+1}^2$ denotes the standard Euclidean metric on $\mathbb{R}^{m+1}$.
The components of the induced Riemannian metric $g=\phi^*h$ of this hypersurface are given by
$g_{ii}=1+a_i^2$ and $g_{ij}=a_i a_j$ for $i\neq j$. We look for
$\lambda(x_1,...,x_m)=\left(\sum_{i=1}^ma_ix_i+b\right)^t,$ for some constant $t$. A straightforward computation yields
\begin{align}\notag
\alpha(\lambda)=|\nabla \lambda|^2 = \frac{t^2|a|^2}{1+|a|^2}\lambda^{2(t-1)/t}, \;\beta(\lambda)=\Delta\lambda = \frac{t(t-1)|a|^2}{1+|a|^2}\lambda^{(t-2)/t},
\end{align}
where $|a|^2=\sum_{i=1}^m a_i^2$.
Substituting these into (\ref{TGCM}) with $c=0$, we conclude that the conformal hypersurface $\phi : (\Omega, \overline{g}=\lambda^{-2}g)\to (\mathbb{R}^{m+1}, h)$ with  $m\neq2$
is proper biharmonic if and only if $$2(m-4)t^2-(m-8)t-2=0.$$
If $m=4$, we get $t=1/2$. For $m\neq4$, we have $t=1/2$ or $t=-2/(m-4)$.
\end{example}

For a hyperbolic space form, i.e.,  a space form with constant negative sectional curvature, we have the following examples of  proper biharmonic conformal totally geodesic hypersurfaces.

\begin{example} ({\em Proper biharmonic conformal hypersurfaces in a hyperbolic space})\\
It is easy to check that the hypersurface
\begin{eqnarray*}
 \phi:(\mathbb{R}^{m-1}\times\mathbb{R}^+,g_{ij}=x_{m}^{-2}\delta_{ij})  &\rightarrow&  (\mathbb{H}^{m+1},h_{ab}=y_{m+1}^{-2}\delta_{ab}),
\end{eqnarray*}
with $ \phi (x_1,...,x_m) = (1,x_1,...,x_m)$, is totally geodesic and Einstein with ${\rm Ric}^M=-(m-1)\,g$. We look for $\lambda=x_m^{-t}$ for some constant $t\neq0$ so that
 the conformal hypersurface
$$\phi:(\mathbb{R}^{m-1}\times\mathbb{R}^+,\overline{g}_{ij}=x_{m}^{2(t-1)}\delta_{ij})  \rightarrow (\mathbb{H}^{m+1},h_{ab}=y_{m+1}^{-2}\delta_{ab}),$$
is proper biharmonic. A straightforward computation yields $\alpha(\lambda)=|\nabla\lambda|^2=t^2\lambda^2$, and $\beta(\lambda)=\Delta \lambda=[(m-1)t+t^2]\lambda$.
Substituting these into (\ref{TGCM}) with $c=-1$, we conclude that the conformal hypersurface 
is proper biharmonic if and only if $$(m-4)t^2-(m-1)t+(m-1)=0,$$ which has solutions for $m=3, 4, 5$. More precisely,
for $m=4$, we get $t=1$ which had been found in \cite{Ou3}. For $m=3, 5$, we have $t=\pm\sqrt{3}-1$ and $t=2$ respectively.
\end{example}

\begin{example} ({\em Proper biharmonic conformal hypersurfaces in a sphere})\\
We know (see \cite{Ou3}) that the conformal hypersurface
\begin{eqnarray*}
  \phi:(\mathbb{S}^4\backslash\{P\}\equiv\mathbb{R}^4,\overline{g}_{ij}=\delta_{ij}) &\rightarrow& (\mathbb{S}^5\backslash\{N\}\equiv\mathbb{R}^5,h_{ab}=\frac{4\delta_{ab}}{(1+|y|^2)^2}),\\
   (x_1,x_2,x_3,x_4)&\mapsto& (x_1,x_2,x_3,x_4,0)
\end{eqnarray*}
is proper biharmonic. Now let us use Corollary \ref{co3} to verify this.  it is easy to check that the associated isometric immersion
\begin{eqnarray*}
  \phi:(\mathbb{S}^4\backslash\{P\}\equiv\mathbb{R}^4,g_{ij}=\frac{4\delta_{ij}}{(1+|x|^2)^2}) &\rightarrow& (\mathbb{S}^5\backslash\{N\}\equiv\mathbb{R}^5,h_{ab}=\frac{4\delta_{ab}}{(1+|y|^2)^2}),
\end{eqnarray*}
is totally geodesic. For $\lambda=2(1+|x|^2)^{-1}$, a straightforward computation yields
$$\alpha(\lambda)=|\nabla \lambda|^2=\frac{4|x|^2}{(1+|x|^2)^2}=2\lambda-\lambda^2,\quad \beta(\lambda)=\Delta\lambda=\frac{4(|x|^2-1)}{1+|x|^2}=4(1-\lambda).$$
So this $\lambda=2(1+|x|^2)^{-1}$ is indeed an isoparametric function, and one can check that it does satisfy Equation (\ref{TGC4}).
\end{example}

\section{Biharmonic conformal immersions of non-minimal totally umbilical hypersurfaces}
It is well known that totally geodesic hypersurfaces are a subclass of totally umbilical hypersurfaces and that a totally umbilical hypersurface is minimal if and only if it is totally geodesic. In this section we study the conditions under which a non-minimal (i.e., non-totally geodesic) totally umbilical  hypersurface can be biharmonically conformally immersed into a space form.

First, we will prove that  if a non-minimal totally umbilical hypersurface in a space form  can be biharmonically conformally immersed into the space form, then the conformal factor $\lambda$  must be an isoparametric function with respect to the induced metric. Furthermore, the explicit formulas for $|\nabla\lambda|^2, \Delta\lambda$ are determined.

\begin{corollary}\label{th2}
A non-minimal totally umbilical hypersurface $\phi : M^m\to (N^{m+1}(c), h)$ can be  biharmonically conformally immersed into $ (N^{m+1}(c), h)$ if and only if there exists a smooth function $\lambda: (M^m, \phi^*h)\to \r^+$ such that
\begin{align}\label{TU1}
&\Big[ 2m(m-4)H^2\lambda-2(m-2) (m-1)(H^2+c)\lambda +(m-2)^2\lambda^{-1}|\nabla  \lambda|^2  \Big] \nabla  \lambda\\\notag
&  -(m-2)\nabla (\lambda\Delta \lambda)+\frac{1}{2}(m-2)(m-4)\nabla |\nabla \lambda|^2=0,\\\label{TU2}
&m^2 (c-H^2)\lambda^2+ 4  \lambda\Delta \lambda-(m^2-8)  \left|\nabla  \lambda\right|^2=0.
\end{align}
\end{corollary}

\begin{proof}
It is well known that  for a totally umbilical hypersurface $(M^m, g)$ in a space form  $(N^{m+1}(c), h)$, we have 
the shape operator $A=H Id$ where $Id$ is the identity map. By using Gauss's equation we find that
\begin{equation}\label{eq-th2}
 {\rm Ric}^M\,(\nabla \lambda)=(m-1)(H^2+c)\nabla \lambda.
\end{equation}
A straightforward computation using  Equation (\ref{eq-th2}), the assumption $H\ne0$, and the identities
\begin{eqnarray*}
   \Delta \ln \lambda=\lambda^{-1}\Delta \lambda-\lambda^{-2}|\nabla \lambda|^2,\;
   |\nabla \ln\lambda|^2=\lambda^{-2}|\nabla \lambda|^2,
\end{eqnarray*}
we obtain (\ref{TU1}) and (\ref{TU2})  respectively from (\ref{NE}) and (\ref{TE}).\\
\end{proof}

It is easy to check that for $m=2$, (\ref{TU1}) and (\ref{TU2}) imply that $\lambda={\rm constant}$ and $H^2=c$. Note that $\lambda={\rm constant}$ means the conformal immersion is actually an isometric immersion up to a homothety. Using the fact that biharmonicity is invariant under homothety and the well known classifications of biharmonic surfaces in a 3-dimensional space form we recover the following 
\begin{corollary} \label{OWC}\cite{Ou2, WC}
Any totally umbilical conformal biharmonic surface in $\r^3,\;{\rm or} \;H^3$  is minimal, and the only conformal biharmonic surface in $S^3$ is actually a biharmonic surface and hence it is a part of $S^2(\frac{1}{\sqrt{2}})$ up to a homothety. 
\end{corollary}
 Note that  any totally umbilical surface in $\r^3$ is a part of a plane or a sphere. So the statement in Corollary \ref{OWC} about totally umbilical biharmonic conformal surface in $\r^3$ is equivalent to saying that no part of a non-minimal totally umbilical surface of $\r^3$ can be biharmonically conformally immersion into $\r^3$, which was proved in Corollary 2.9 in \cite{Ou2}.\\
 
Our next theorem  shows that the conformal factor of a   non-minimal totally  umbilical biharmonic conformal  hypersurface is an  isoparametric function whose $|\nabla\lambda|^2, \Delta\lambda$ can  be  determined explicitly.

\begin{theorem}\label{prop1}
If a non-minimal totally umbilical hypersurface $\phi : M^m\to (N^{m+1}(c), h)$  $(m\ge  3)$ in a space form can be biharmonically conformally immersed into $(N^{m+1}(c), h)$ with a conformal factor $\lambda$, then  $\lambda$ is an isoparametric function on  $ (M^m, \phi^*h)$ with 
\begin{eqnarray}
|\nabla \lambda|^2  &=&  \delta\,\lambda^{2} +\frac{2C_0}{m(m-2)}\,\lambda^{\frac{4}{m}},\label{function-a}\\
\Delta \lambda &=&  \frac{1}{4}\left[(m^2-8)\delta+m^2(H^2-c)\right]\lambda+\frac{C_0(m^2-8)}{2m(m-2)}\lambda^{-\frac{m-4}{m}},\label{b}
\end{eqnarray}
for some constant $C_0$ and $\delta=-\frac{{m}^{3}-2\,{m}^{2}+4\,m+8}{(m-2)^3} {H}^{2}+ c$.
\end{theorem}

\begin{proof}
It is easily checked that (\ref{TU1}) can be written as
\begin{align}\label{TU10}
&\Big[ 4m\frac{m-4}{m-2}H^2\lambda-4(m-1)(H^2+c)\lambda +2(m-2)\lambda^{-1}|\nabla  \lambda|^2  \Big] \nabla  \lambda\\\notag
 = &\nabla  [2 \lambda\Delta \lambda-(m-4) |\nabla \lambda|^2].
\end{align}
\textbf{Case I:} if $2\lambda \Delta\lambda-(m-4)|\nabla \lambda|^2=C_1$, a constant, then (\ref{TU10})  implies that  either $\nabla  \lambda\equiv 0$ and hence $\lambda$ is a constant, or
\begin{align}\label{xx}
4m\frac{m-4}{m-2}H^2\lambda-4(m-1)(H^2+c)\lambda +2(m-2)\lambda^{-1}|\nabla  \lambda|^2 =0,
\end{align} on an open set. In the latter case, we conclude from (\ref{xx}) that $|\nabla  \lambda|^2$ is a polynomial function of $\lambda$.  From  this  and $2\lambda \Delta\lambda-(m-4)|\nabla \lambda|^2=C_1$  we conclude that $\Delta \lambda$ is  also a function in $\lambda$. So, $\lambda$ is also an isoparametric function in this case.\\
\textbf{Case II:} if $2 \lambda \Delta\lambda-(m-4)|\nabla \lambda|^2$ is not a constant, then there exist a point $x_0$ such that  $\nabla \,[ 2 \lambda \Delta\lambda-(m-4)|\nabla \lambda|^2](x_0)\ne 0$. Applying Lemma 2 of \cite{BO} we have
\begin{align}\label{TG5}
2\lambda \Delta\lambda-(m-4)|\nabla \lambda|^2=u(\lambda),
\end{align}
for some nonconstant smooth function $u$ such that
\begin{align}\label{TG6}
u'(\lambda)=4m\frac{m-4}{m-2}H^2\lambda-4(m-1)(H^2+c)\lambda +2(m-2)\lambda^{-1}|\nabla  \lambda|^2.
\end{align}
This, together with Equations (\ref{TG5}), shows that  $\lambda$ is also an isoparametric function in this case. This gives the first statement of the theorem.\\

Now, since $H\ne 0$,  we can use (\ref{TG5}) and (\ref{TU2}) to have
\begin{align}\label{TG9}
&\left|\nabla \lambda\right|^2=\frac{1}{m(m-2)}\left[m^2(c-H^2)\lambda^2+ 2u(\lambda)\right].
\end{align}
By substituting (\ref{TG9}) into (\ref{TG6}), we obtain
\begin{align}\notag
u'(\lambda)=-2\left[\frac{m^2+4}{m-2}H^2+(m-2)c\right]\lambda+\frac{4}{m}\lambda^{-1}u(\lambda).
\end{align}
By solving this first order linear differential equation we get

\begin{eqnarray}\label{TG12}
  u(\lambda)
   &=& -m\left[\frac{m^2+4}{(m-2)^2}H^2+\,c\right]\lambda^2+C_0 \lambda^{\frac{4}{m}},
\end{eqnarray}
for some constant $C_0$. Substituting (\ref{TG12}) into  (\ref{TG9}) yields
\begin{eqnarray}\label{TG13}
|\nabla\lambda|^2&=&\delta\,{\lambda}^{2} +\frac{2C_0}{m(m-2)}\,{\lambda}^{\frac{4}{m}},
\end{eqnarray}
where  $\delta=-\frac{{m}^{3}-2\,{m}^{2}+4\,m+8}{(m-2)^3} {H}^{2}+ c$.

Substituting (\ref{TG12}), (\ref{TG13}) into  (\ref{TG5}) and solving for $\Delta \lambda$ we obtain  (\ref{b}),

which completes the proof  of  the  theorem.
\end{proof}

As an application  of  Theorem \ref{prop1}, we prove  the  following  proposition which can  be stated as any  totally umbilical conformal biharmonic hypersurface in a space form of non-positive sectional curvature is  minimal.

\begin{proposition}\label{4-DN}
No part of a non-minimal totally umbilical hypersurface in a space form $N^{m+1}(c)$ with $c\le 0$ and $m\ge 4$ admits a  biharmonic conformal immersion into its ambient space $N^{m+1}(c)$.
\end{proposition}

\begin{proof}
Note that if the conformal factor is a constant, then a biharmonic conformal immersion is actually an isometric immersion up to a homothety. It is well known that any totally umbilical biharmonic hypersurface in a space form of nonpositive curvature is minimal, which contradicts our assumption. So, we may assume  that $\phi : (U, g)\to (N^{m+1}(c), h)$ is a non-minimal totally umbilical hypersurface and  that there exists a non-constant positive function $\lambda$ such that the
conformal immersion of hypersurface $\phi : (U, \bar{g}=\lambda^{-2}g)\to (N^{m+1}(c), h)$ is biharmonic.

 It is not difficult to check  that for a totally umbilical hypersurface in a space form of constant sectional curvature $c$, we have
\begin{align}\label{ric}
 &{\rm Ric}^M\,(\nabla\lambda,\nabla\lambda)=  (m-1)(H^2+c)|\nabla\lambda|^2,\\\label{T20}
& |\nabla \lambda|^2  =\alpha(\lambda)=  \delta\,{\lambda}^{2} +k\,{\lambda}^{\frac{4}{m}},\\\label{T21}
& \Delta \lambda =\beta(\lambda)=  \frac{1}{4}\left[(m^2-8)\delta+m^2(H^2-c)\right]\lambda+\frac{1}{4}(m^2-8)k\lambda^{-\frac{m-4}{m}},
\end{align}
for some constant $k$ and $\delta=-\frac{{m}^{3}-2\,{m}^{2}+4\,m+8}{(m-2)^3} {H}^{2}+ c$. A straightforward computation yields
\begin{eqnarray}
\Delta|\nabla\lambda|^2 &=&  \alpha'(\lambda)\beta(\lambda)+\alpha(\lambda)\alpha''(\lambda),\label{T5-1}\\
g(\nabla\lambda,\nabla\Delta\lambda)&=& \alpha(\lambda)\beta'(\lambda),\label{T5-2}\\
{\rm Ric}(\nabla\lambda,\nabla\lambda)&=& (m-1)(H^2+c)\alpha(\lambda).\label{T5-3}
\end{eqnarray}
Substituting (\ref{ric})-(\ref{T5-3}) into the  Bochner-Weitzenb\"{o}ck formula
$$\frac{1}{2}\Delta|\nabla\lambda|^2=|{\rm Hess\,}\lambda|^2+g(\nabla\lambda,\nabla\Delta\lambda)+{\rm Ric}(\nabla\lambda,\nabla\lambda),$$
and using the Newton's inequality $|{\rm Hess\,}\lambda|^2\geq \frac{1}{m}(\Delta\lambda)^2$ we have
\begin{eqnarray}\label{T5-4}
 \frac{1}{2}\alpha'(\lambda)\beta(\lambda)  &+& \frac{1}{2}\alpha(\lambda)\alpha''(\lambda) -\alpha(\lambda)\beta'(\lambda)-(m-1)(H^2+c)\alpha(\lambda) \\  \nonumber
   &-&  \frac{1}{m}\beta(\lambda)^2\geq0.
\end{eqnarray}
A further computation using (\ref{T20}) and (\ref{T21})  we can rewrite (\ref{T5-4}) as
\begin{eqnarray}\label{T5-5}
  Ak^2\lambda^{-\frac{2(m-4)}{m}}+Bk\lambda^{\frac{4}{m}}+C\lambda^2\geq0,
\end{eqnarray}
where $A$, $B$, and $C$ are given by
\begin{eqnarray*}
  A &=& -(m-4)(m-2)^6(m^4-8m^2+32), \\
  B   &=& 8 c (m-2)^6(m^4-2m^3-12m+16)\\
      &&-32 (m-2)^3\left( {m}^{4}+10\,{m}^{3}-20\,{m}^{2}-8\,m+32\right){H}^{2},\\
  C &=& -16\,{c}^{2}m \left( m-2 \right) ^{6} \left( {m}^{2}-2\,m+4 \right) \\
     &&   -32\,mc \left( m-2 \right) ^{3} \left( {m}^{4}-4\,{m}^{3}+8\,{m}^{2}-32\right) {H}^{2}\\
 &&+256\, \left({m}^{6}-2{m}^{5}+2{m}^{4}+8\,{m}^{3}-
8\,{m}^{2}-16\,m \right) {H}^{4}.
\end{eqnarray*}
For $m\geq4$,  we use that fact that $A\leq0$ and  (\ref{T5-5}) to have
\begin{eqnarray}\label{T5-6}
  Bk\lambda^{\frac{4}{m}}+C\lambda^2\geq0.
\end{eqnarray}
On the other hand, since $\lambda$ is not constant we may assume $ |\nabla \lambda|^2  >0$ in an open set, this and  (\ref{function-a}) implies that 
\begin{align}\notag
 k\,{\lambda}^{\frac{4}{m}}\geq-\delta\,{\lambda}^{2}, 
 \end{align}
 that is
\begin{eqnarray}\label{T5-7}
  k\,{\lambda}^{\frac{4}{m}}\geq\left(\frac{{m}^{3}-2\,{m}^{2}+4\,m+8}{(m-2)^3} {H}^{2}- c\right)\,{\lambda}^{2}.
\end{eqnarray}
Since $c\leq0$, one can easily check that  $B\leq0$. Combining (\ref{T5-6}) and (\ref{T5-7}) yields
\begin{eqnarray}\label{T5-8}
  B\left(\frac{{m}^{3}-2\,{m}^{2}+4\,m+8}{(m-2)^3} {H}^{2}- c\right)+C\geq0.
\end{eqnarray}
By substituting the values of $B$ and $C$ into (\ref{T5-8}), we have
\begin{eqnarray*}
  &&-8\,{c}^{2} \left( m-2 \right) ^{6} \left( {m}^{4}-4\,{m}^{2}-4\,m+16\right)\\
  &&+8\,c \left( m-2 \right) ^{3} \left( {m}^{7}-4\,{m}^{6}+4\,{m}^{5}+8\,{m}^{4}+32\,{m}^{3}-160\,{m}^{2}+64\,m+256 \right)H^2\\
  &&-32\left({m}^{7}-20\,{m}^{5}+64\,{m}^{4}-16\,{m}^{3}-192\,{m}^{2}+192\,m+256\right)H^4\geq0.
\end{eqnarray*}
which is a contradiction since each of the three summands is negative due to the fact that $m\geq4$, $c\leq0$, and
\begin{eqnarray*}
  &&{m}^{4}-4\,{m}^{2}-4\,m+16\geq {m}^{3}(m-4) +4\,m (m-3)>0 ,\\
  && {m}^{7}-4\,{m}^{6}+4\,{m}^{5}+8\,{m}^{4}+32\,{m}^{3}-160\,{m}^{2}+64\,m+256\\
   &&   \geq {m}^{5}(m-2)^{2}+8\,{m}^{2}(m-3)^{2}>0,\\
  &&{m}^{7}-20\,{m}^{5}+64\,{m}^{4}-16\,{m}^{3}-192\,{m}^{2}+192\,m+256\\
  &&\geq(m-3)\left[{m}^{4}({m}^{2}-9) +{m}^{4}(m-2)\right]>0.
\end{eqnarray*}
The contradiction completes the proof of the proposition.
\end{proof}


In contrast with Proposition \ref{4-DN}, our next proposition shows that when the ambient space has positive sectional curvature, we  do have  non-minimal totally umbilical biharmonic conformal hypersurfaces  in  $\mathbb{S}^{5}$. 
\begin{proposition}\label{S4}
For  $a, b\in \r$ with  $a>0$ and $a^2+b^2=1$.
A part of totally umbilical hypersurface $\phi:\mathbb{S}^m(a)\to \mathbb{S}^{m+1}, \phi(v)=(v, b)$ can be biharmonically conformally  immersed into $\mathbb{S}^{m+1}$ with $\lambda=(1+|x|^2)^{-1}$, where $x=(x_1, \cdots, x_m) $ is the conformal coordinates on $S^m(a)$,  if and only if $m=4$ and $a=1, b=0$, or $a=\frac{\sqrt{3}}{2}$ and $b=\pm  \frac{1}{2}$.
\end{proposition}

\begin{proof}
One can  check (see, e.g., \cite{PB}) that the totally umbilical hypersurface
$$\phi: \mathbb{S}^m\left(a\right)\to \mathbb{S}^{m+1}, \phi(v)=(v, b)$$
has shape operator and the mean curvature given by
\begin{eqnarray*}
  A(X) = - \frac{b}{a} X,\quad\forall X\in\Gamma(TM),\;\;\;
  H  =- \frac{b}{a}.
\end{eqnarray*}
One can also check that  by using the conformal coordinates induced by stereographic projection the totally  umbilical hypersurface  can be  expressed as
\begin{eqnarray*}
 \phi:(\r^m,g=\frac{4a^2\delta_{ij}}{(1+|x|^2)^2})  &\rightarrow& (\mathbb{S}^{m+1},h=\frac{4\delta_{\alpha\beta}}{(1+|y|^2)^2}). \\
   (x_1,...,x_m)&\mapsto& \frac{a}{1-b}\left(\frac{2x_1}{1+|x|^2},...,\frac{2x_m}{1+|x|^2},\frac{-1+|x|^2}{1+|x|^2}\right).
\end{eqnarray*}

By definition, this hypersurface can  be biharmoncally conformally immersed into $S^{m+1}$ with conformal factor $\lambda=(1+|x|^2)^{-1}$ means that the  the conformal immersion $\phi : (\r^m, \lambda^{-2}g=4a^2\delta_{ij})\to (\mathbb{S}^{m+1},h)$ is biharmonic. Note that,  up  to a homothety, this map is the composition of the inverse stereographic projection  $p^{-1}:\r^m \to  S^m(a)$ followed  by the standard embedding  $S^m(a)\hookrightarrow S^{m+1}, v\mapsto (v, b)$.

Note (cf. e.g.,   \cite{Ou8}) also that  the inverse stereographic projection  $p^{-1}:\r^m \to  S^m(a)$, as a conformal map, is biharmonic if and only if $m=4$. On the other hand, it was proved  in \cite{S} that the composition of a biharmonic map followed by a totally geodesic map is again a  biharmonic map. From  these, together  with  the observation  that the map corresponding  to the  case $m=4, a=1, b=0$ is  the composition of a biharmonic  conformal map followed by a  totally geodesic embedding, we obtain the first part of the statement of the proposition.

Now we will  use Theorem \ref{prop1} to prove the second part of the statement. We use  $\lambda(x)=(1+|x|^2)^{-1}$ and a direct calculation to  have
\begin{eqnarray*}
  \nabla\lambda = -\frac{1}{2a^2}\,x_i\frac{\partial}{\partial x_i}, \;\;
  |\nabla\lambda|^2 = \frac{1}{a^2}\,\frac{|x|^2}{(1+|x|^2)^{2}},\;\;
  \Delta\lambda   =  \frac{m}{2a^2}\frac{-1+|x|^2}{1+|x|^2}.
\end{eqnarray*}
A further  computation yields
\begin{eqnarray*}
|\nabla\lambda|^2 =\frac{1}{a^2}(\lambda-\lambda^2), \;\;\;
\Delta\lambda =  \frac{m}{2a^2}(1-2\lambda).
\end{eqnarray*}
By comparing these with (\ref{function-a})  and (\ref{function-b}) respectively,  and using $c=1$, we have
\begin{eqnarray*}
&& \delta=-\frac{1}{a^2},\;\; m=4,\;\; \frac{2C_0}{m(m-2)}=\frac{1}{a^2},\\
&&  \frac{1}{4}\left[(m^2-8)\delta+m^2(\frac{b^2}{a^2}-1)\right]= -\frac{m}{a^2},\;\;\frac{m}{2a^2}=\frac{C_0(m^2-8)}{2m(m-2)},
\end{eqnarray*}
Solving these equations we have the only solution $m=4$, $b=\pm\frac{1}{2}$ and $a=\frac{\sqrt{3}}{2}$,  which completes the proof of proposition.
\end{proof}
\begin{remark}
Note that, up to a homothety, the two proper biharmonic conformal immersions given in  Proposition  \ref{S4}  can be viewed as the compositions
 \begin{center}$ \r^4 \xrightarrow[projection]{\tiny inverse\,stereographic\;} S^4\;\;\;\;\;\;\;\;\; \xrightarrow[embedding]{totally\, geodesic\;\;} S^5$,\\
 $ \r^4 \xrightarrow[projection]{\tiny inverse\,stereographic\;} S^4\footnotesize{\left(\frac{\sqrt{3}}{2}\right)} \xrightarrow[embedding]{totally\, umbilical\;} S^5$.\\
  \end{center}
 The interesting  thing is that this method works only for  dimension $m=4$.
\end{remark}



\begin{thebibliography}{99}
\bibitem{AO} A. M.  Akyol and Y.-L. Ou, {\em  Biharmonic Riemannian submersions}, Ann. Mat. Pura   Appl. (4) 198 (2019), no. 2, 559--570.
\bibitem{B2} P. Baird, {\em Stress-energy tensors and the Lichnerowicz Laplacian}, J. Geom. Phys. 58 (2008), no. 10, 1329--1342.
\bibitem{BK} P. Baird and D. Kamissoko, {\em On constructing biharmonic maps and metrics}, Ann. Global Anal. Geom.
23 (2003), no. 1, 65--75.
\bibitem{BFO} P. Baird, A. Fardoun, and S. Ouakkas, {\em Conformal and semi-conformal biharmonic maps}. Ann. Glob. Anal. Geom. 34, 403--414 (2008).
\bibitem{BLO} P. Baird, E. Loubeau, and C. Oniciuc, {\em Harmonic and biharmonic maps from surfaces}, Contemp. Math., Amer. Math. Soc., Providence, RI, 542 (2011), 223--230.
\bibitem{BO} P. Baird and Y.-L. Ou, {\em Biharmonic conformal maps in dimension four and equations of Yamabe-type}, J. Geom. Anal., 28(4) (2018), 3892--3905.
\bibitem{GO} E. Ghandour  and Y.-L. Ou, {\em Generalized harmonic morphisms and horizontally weakly conformal biharmonic maps}, J. Math. Anal. Appl. 464 (2018), no. 1, 924--938.
\bibitem{Ji1} G. Y. Jiang, {\em $2$-harmonic maps and their first and second variational formulas}. Chinese Ann. Math. Ser. A, 7 (1986), 389--402.
\bibitem{LO1} E. Loubeau and Y.-L. Ou {\em  The characterization of biharmonic morphisms},  Diff. geom. and its appl. (Opava, 2001), 31--41, Math. Publ., 3, Silesian Univ. Opava, Opava, 2001.
\bibitem{LO2} E. Loubeau and Y.-L. Ou {\em Biharmonic maps and morphisms from conformal mappings}, Tohoku Math. J. (2) 62 (2010), no. 1, 55--73.
\bibitem{MS} S. Maeta and M. Shito, {\em Classification of biharmonic Riemannian submersions from manifolds with constant sectional curvature}, preprint 2025, arXiv:2509.05939.
\bibitem{MM} A. Mohammed Cherif, K. Mouffoki, {\em $p$-Biharmonic hypersurfaces in Einstein space and conformally flat space}, Bull. Korean Math. Soc., 60 (3) (2023), 705--715.
\bibitem{O1} C. Oniciuc, {\em Biharmonic maps between Riemannian manifolds},  An. Stii, Al. Univ. ``Al. I. Cuza'' Iasi, 68 (2002), 237--248.
\bibitem{O2} C. Oniciuc, {\em Biharmonic submanifolds in space forms}, Habilitation Thesis (2012), www.researchgate.net, https://doi.org/10.13140/2.1.4980.5605.
\bibitem{Ou0} Y. -L. Ou, {\em Biharmonic morphisms between Riemannian manifolds}, Geom. and topology of submanifolds, X (Beijing/Berlin, 1999), 231--239, World Sci. Publ., River Edge, NJ, 2000.
\bibitem{Ou3} Y. -L. Ou, {\em On conformal biharmonic immersions }, Anal. Global Analysis and Geom. 36 (2009), 133--142.
\bibitem{Ou1} Y. -L. Ou, {\em Biharmonic hypersurfaces in Riemannian manifolds}, Pacific J. of Math, 248 (1) (2010), 217--232.
\bibitem{Ou2} Y. -L. Ou, {\em Biharmonic conformal immersions into three-dimensional manifolds}, Mediterr. J. Math. 12, (2015), 541--554.
\bibitem{Ou6}  Y. -L. Ou, {\em Some recent progress of biharmonic submanifolds}.  Contemp. Math., Amer. Math. Soc., Providence, RI,  674 (2016), 127--139.
\bibitem{Ou4} Y. -L. Ou, {\em $f$-Biharmonic maps and $f$-biharmonic submanifolds II},  J. Math. Anal. Appl., 455 (2017), 1285-1296.
\bibitem{Ou5} Y. -L. Ou, {\em Some recent work on biharmonic conformal maps}, Contemp. Math., Amer. Math. Soc., Providence, RI, 756 (2020), 195-205. 
\bibitem{OC} Y. -L. Ou and B. -Y. Chen, {\em Biharmonic submanifolds and biharmonic maps in Riemannian Geometry}, World Scientific Publishing Co Pte Ltd,  May 2020.
\bibitem{Ou7} Y. -L. Ou, {\em  A short survey on biharmonic Riemannian submersions}, Int. Electron. J. Geom. 17 (2024), no. 1, 259--266. 
\bibitem{Ou8} Y. -L. Ou, {\em  Some classifications of conformal biharmonic and k-polyharmonic maps}. Front. Math. 18 (2023), no. 1, 1--15. 
 \bibitem{PB} O. Perdomo and  and A. Brasil Jr., {\em Stability index jump for constant mean curvature hypersurfaces of spheres}, Arch. Math. 99 (2012), 493--500.
 \bibitem{S} H. Sun, {\em A theorem on 2-harmonic mappings}, J. Math. (China), 12 (1) (1992), 103--106.
\bibitem{U1} H. Urakawa, {\em Harmonic maps and biharmonic maps on principal bundles and warped products},  J. Korean Math. Soc., 55(3), (2018), 553-574.
\bibitem{U2} H. Urakawa, {\em Harmonic maps and biharmonic Riemannian submersions}, Note di Mate. 39 (1) (2019), 1--23.
\bibitem{WC} Z. -P. Wang and   X. -Y. Chen,   {\em  Biharmonic conformal immersions into a 3-dimensional conformally flat space}, preprint 2024,  arXiv:2408.10144.
\bibitem{WO1} Z. -P. Wang and   Y. -L. Ou,  {\em Biharmonic Riemannian submersions from 3-manifolds},  Math.  Zeitschrift, 269 (3) (2011), 917-925.
\bibitem{WO2} Z. -P. Wang and   Y. -L. Ou,   {\em Biharmonic Riemannian submersions from a 3-dimensional BCV space}, J Geom Anal 34, 63 (2024). https://doi.org/10.1007/s12220-023-01501-9.
\bibitem{WO3} Z. -P. Wang and   Y. -L. Ou,   {\em Biharmonic isometric immersions into and biharmonic Riemannian submersions from $M^2\times \r$},  J Geom. Anal. 35, 20 (2025). https://doi.org/10.1007/s12220-024-01828-x
\end{thebibliography}
\end{document}